\documentclass[letterpaper, 12pt]{article}
\usepackage{amssymb,amsmath,amsthm}
\usepackage{epsf}
\usepackage{times,bm,mathrsfs}  
\usepackage{amsmath}
\usepackage{amssymb}
\usepackage{amsfonts}
\usepackage{mathrsfs}
\usepackage{bbm}
\usepackage{color}
\usepackage{graphicx}
\usepackage{amscd}
\usepackage{epsfig}

\definecolor{Red}{rgb}{1,0,0}
\definecolor{Blue}{rgb}{0,0,1}
\definecolor{Olive}{rgb}{0.41,0.55,0.13}
\definecolor{Green}{rgb}{0,1,0}
\definecolor{MGreen}{rgb}{0,0.8,0}
\definecolor{DGreen}{rgb}{0,0.55,0}
\definecolor{Yellow}{rgb}{1,1,0}
\definecolor{Cyan}{rgb}{0,1,1}
\definecolor{Magenta}{rgb}{1,0,1}
\definecolor{Orange}{rgb}{1,.5,0}
\definecolor{Violet}{rgb}{.5,0,.5}
\definecolor{Purple}{rgb}{.75,0,.25}
\definecolor{Brown}{rgb}{.75,.5,.25}
\definecolor{Grey}{rgb}{.5,.5,.5}
\definecolor{Black}{rgb}{0,0,0}

\def\path{{\tt path}}

\newcommand{\ecal}{\mathcal{E}}

\newcommand{\tcal}{\mathcal{T}}

\newcommand{\eps}{\varepsilon}

\newcommand{\ind}{\mathbbm{1}}

\newcommand{\bdm}{\begin{displaymath}}
\newcommand{\edm}{\end{displaymath}}
\newcommand{\bea}{\begin{eqnarray*}}
\newcommand{\eea}{\end{eqnarray*}}
\newcommand{\bean}{\begin{eqnarray}}
\newcommand{\eean}{\end{eqnarray}}

\newcommand{\expec}{\mathbb{E}}
\newcommand{\E}{\mathbb{E}}

\newtheorem{theorem}{Theorem}
\newtheorem{proposition}{Proposition}
\newtheorem{claim}{Claim}

\newtheorem{lemma}{Lemma}

\newtheorem{remark}{Remark}

\renewcommand{\P}{\mathbb{P}}

\newcommand{\Q}{\mathbb{Q}}

\begin{document}

\title{Distance-based species tree estimation
	under the coalescent: information-theoretic trade-off between number of loci and sequence length\thanks{
S.R.~is supported by NSF grants DMS-1007144, DMS-1149312 (CAREER), and an Alfred P. Sloan Research Fellowship. 
E.M.~is supported by NSF grants DMS-1106999 and CCF 1320105
and DOD ONR grants  N000141110140 and N00014-14-1-0823 and grant
328025 from the Simons Foundation. E.M.~and S.R.~thank the Simons Institute for
the Theory of Computing at U.C. Berkeley where
this work was done.
}
}
\author{
{\bf Elchanan Mossel}\\
       {MIT}\\
       {\small \texttt{elmos@mit.edu}}
       {}\\
       {}\\
{\bf Sebastien Roch}\\
       {UW Madison}\\
       {\small \texttt{roch@math.wisc.edu}}
}
\maketitle
\thispagestyle{empty}

\begin{abstract}
We consider the reconstruction of a phylogeny
from multiple genes under the multispecies coalescent.
We establish a connection with the sparse signal detection
problem, where one seeks to distinguish between
a distribution and a mixture of the distribution 
and a sparse signal. Using this connection,
we derive an information-theoretic trade-off
between the number of genes, $m$, needed for an accurate
reconstruction and the sequence length, $k$, of the
genes. Specifically, we show that to detect
a branch of length $f$, one needs $m = \Theta(1/[f^{2} \sqrt{k}])$ genes.
\end{abstract}

\clearpage

\section{Introduction}

In the {\bf sparse signal detection problem}, 
one is given $m$ i.i.d.~samples $X_1, \ldots, X_m$
and the goal is 
to distinguish between a distribution $\P_0^{(m)}$
$$
H_0^{(m)}: X_i \sim \P_0^{(m)},
$$
and the same distribution corrupted by a 
sparse signal $\P_1^{(m)}$
$$
H_1^{(m)}: X_i \sim \Q^{(m)} := (1-\sigma_m)\,\P_0^{(m)} + \sigma_m\, \P_1^{(m)}.
$$
Typically one takes $\sigma_m = m^{-\beta}$,
where $\beta \in (0,1)$.
This problem arises
in a number of applications~\cite{Dobrushin:58,JeCaLi:10,CaJiTr:05,KHH+:05}. The Gaussian case in particular is well-studied~\cite{Ingster:97,DonohoJin:04,CaJeJi:11}.
For instance it is established in~\cite{Ingster:97,DonohoJin:04} that, in the case
$\P_0^{(m)} \sim N(0,1)$ and $\P_1^{(m)} \sim N(\lambda_m,1)$
with $\lambda_m = \sqrt{2 r \log m}$, a test
with vanishing error probability exists if and
only if $r$ exceeds an explicitly known
{\em detection boundary} $r^*(\beta)$.

In this paper, we establish a connection between
sparse signal detection and the reconstruction of phylogenies
from multiple genes or loci under the multispecies coalescent, a standard population-genetic
model~\cite{RannalaYang:03}. 
The latter problem is of
great practical interest
in computational evolutionary biology
and is currently the subject of intense study.
See e.g.~\cite{LYK+:09,DegnanRosenberg:09,ALP:+:12,Nakhleh:13} for surveys.
The problem
is also related to the reconstruction
of demographic history in population genetics~\cite{MyFePa:08,BhaskarSong:14,KMR+:15}.

By taking advantage of the connection
to sparse signal detection, 
we derive a ``detection boundary'' for the 
multilocus phylogeny problem
and use it to
characterize the trade-off between the
number of genes needed to accurately reconstruct a phylogeny
and the quality of the signal that can be extracted from each
separate gene. Our results apply to distance-based methods, an important class of
reconstruction methods.
Before stating our results more formally, we begin with some background. See e.g.~\cite{SempleSteel:03} for a more general introduction
to mathematical phylogenetics.

\paragraph{Species tree estimation} 
An evolutionary tree, or phylogeny, is a graphical
representation of the evolutionary relationships
between a group of species. Each leaf in the tree corresponds
to a current species while internal vertices
indicate past speciation events.  
In the classical phylogeny estimation problem,
one sequences a {\em single} common gene (or other {\em locus} such as pseudogenes, introns, etc.) from a representative individual of each species of interest. One then seeks to reconstruct the
phylogeny by comparing the genes across species. 
The basic principle is simple:
because mutations accumulate
over time during evolution, more distantly related species 
tend to exhibit more differences between
their genes. 

Formally, phylogeny estimation
boils down to {\em learning the structure of a
latent tree graphical model from i.i.d.~samples
at the leaves}.
Let $T = (V,E,L,r)$ be a rooted leaf-labelled
binary tree, with $n$ leaves denoted by $L = \{1,\ldots,n\}$
and a root denoted by $r$. 
In the Jukes-Cantor model~\cite{JukesCantor:69},
one of the simplest Markovian models of molecular
evolution, we associate to each edge $e \in E$ a mutation
probability 
\begin{equation}\label{eq:mutation-probability}
p_e = 1 - e^{-\nu_e t_e},
\end{equation}
where $\nu_e$ is the mutation rate and $t_e$
is the time elapsed along the edge $e$. 
(The analytical form of~\eqref{eq:mutation-probability}
derives from a continuous-time Markov process
of mutation along the edge. See e.g.~\cite{SempleSteel:03}.)
The {\em Jukes-Cantor process} is defined as follows:
\begin{itemize}
	\item Associate to the root a sequence $\mathbf{s}_r
	= (s_{r,1},\ldots,s_{r,k}) \in \{\mathtt{A}, \mathtt{C}, \mathtt{G}, \mathtt{T}\}^k$
	of length $k$ where each site $s_{r,i}$ is 
	uniform in $\{\mathtt{A}, \mathtt{C}, \mathtt{G}, \mathtt{T}\}$. 
	\item Let $U$ denote the set of children of the root. 
	\item Repeat until $U = \emptyset$: 
	\begin{itemize}
		\item Pick a $u \in U$.
		\item Let $u^-$ be the parent
		of $u$. 
		\item Associate a sequence $\mathbf{s}_u
		\in \{\mathtt{A}, \mathtt{C}, \mathtt{G}, \mathtt{T}\}^k$
		to $u$ as follows: $\mathbf{s}_u$ is obtained from $\mathbf{s}_{u^-}$
		by mutating each site in $\mathbf{s}_{u^-}$ independently
		with probability $p_{(u^-,u)}$; when a mutation occurs
		at a site $i$, replace $s_{u,i}$ with a uniformly
		chosen state in $\{\mathtt{A}, \mathtt{C}, \mathtt{G}, \mathtt{T}\}$.
		\item Remove $u$ from $U$ and add the children (if any) of
		$u$ to $U$.
	\end{itemize}
\end{itemize}
\noindent Let $T^{-r}$ be the tree $T$ where the root is
suppressed, i.e., where the two edges adjacent to the
root are combined into a single edge.
We let $\mathcal{L}[T,(p_e)_e,k]$ be the distribution
of the sequences at the {\em leaves} $\mathbf{s}_1,\ldots,\mathbf{s}_n$ under the Jukes-Cantor process.
We define the {\bf single-locus phylogeny estimation problem} as follows:
\begin{quote}
Given sequences 
at the leaves $(\mathbf{s}_1,\ldots,\mathbf{s}_n) \sim \mathcal{L}[T,(p_e)_e,k]$,
recover the (leaf-labelled) unrooted tree $T^{-r}$.
\end{quote}
(One may also be interested in estimating the
$p_e$s, but we focus on the tree. The root is in
general not identifiable.)
This problem has a long history in evolutionary biology. 
A large number of estimation techniques
have been developed. See e.g.~\cite{Felsenstein:04}.
For a survey of the learning perspective
on this problem, see e.g.~\cite{MSZ+:14}.
On the theoretical side, much is known
about the sequence length---or, in other words, the number of samples---required 
for a perfect reconstruction with high probability, including both information-theoretic lower bounds~\cite{SteelSzekely:03,Mossel:03,Mossel:04a,MoRoSl:11} and matching algorithmic upper bounds~\cite{ErStSzWa:99a,DaMoRo:11a,DaMoRo:11b,Roch:10}. More general models of molecular evolution
have also been considered in this context; see e.g.~\cite{ErStSzWa:99b,CrGoGo:02,MosselRoch:05,DaskalakisRoch:10,AnDaHaRo:10}.

Nowadays, it is common for biologists to have
access to {\em multiple} genes---or even full genomes. This abundance of data, which on the surface
may seem like a blessing, in fact comes with significant challenges. See e.g.~\cite{DeBrPh:05,Nakhleh:13} for surveys. One important issue is that
different genes
may have incompatible evolutionary histories---represented
by incongruent gene trees. In other words, if one
were to solve the phylogeny estimation problem {\em separately}
for several genes, one may in fact obtain {\em different}
trees. Such incongruence can be explained in some
cases by estimation error, but it can also
result from deeper biological processes such as horizontal
gene transfer, gene duplications and losses, and
incomplete lineage sorting~\cite{Maddison:97}. The latter
phenomenon, which will be explained in Section~\ref{section:definitions},
is the focus of this paper. 

Accounting for this type of complication necessitates 
a {\em two-level hierarchical model} for the input data. Let $S = (V,E, L, r)$ be a rooted
leaf-labelled binary
{\em species tree}, i.e., a tree representing
the actual succession of past divergences for
a group of organisms. To each gene $j$ shared by all
species under consideration, we associate
a {\em gene tree} $T_j = (V_j,E_j,L)$,
mutation probabilities $(p^j_e)_{e \in E_j}$,
and sequence length $k_j$.
The triple $(T_j,(p^j_e)_{e \in E_j},k_j)$ is
picked at random according to a given distribution
$\mathcal{G}[S,(\nu_e,t_e)_{e \in E}]$
which depends on the species tree, mutation
parameters $\nu_e$ and inter-speciation times
$t_e$.  
It is standard 
to assume that the gene trees are 
conditionally independent given the species tree.
In the context of incomplete lineage sorting,
the distribution of the gene trees,  $\mathcal{G}$,
is given by the so-called 
{\em multispecies coalescent}, 
which is a canonical model
for combining speciation history and population genetic effects~\cite{RannalaYang:03}. 
(Readers familiar with the multispecies coalescent may observe that our model is a bit richer than the standard model, as it includes mutational parameters in addition to branch length information. Note that we also incorporate sequence length in the model.) 
The
detailed description of the model is deferred to Section~\ref{section:definitions}, 
as it is not needed for
a high-level overview of our results. 
For the readers not familiar
with population genetics, it is useful to think of $T_j$ as a noisy
version of $S$ (which, in particular, may result in $T_j$ having a different (leaf-labelled) topology than $S$).


Our two-level model of sequence data is then as follows.
Given a species tree $S$, parameters $(\nu_e,t_e)_{e\in E}$ and a number of genes $m$:
\begin{enumerate}
	\item {\bf [First level: gene trees]} Pick $m$ independent gene trees and parameters
	\begin{displaymath}
	(T_j,(p^j_e)_{e \in E_j},k_j) \sim \mathcal{G}[S,(\nu_e,t_e)_{e \in E}], \qquad j =1,\ldots, m. 
	\end{displaymath}

	\item {\bf [Second level: leaf sequences]} For each gene $j = 1,\ldots,m$, generate
	sequence data at the leaves $L$ according to the
	(single-locus) Jukes-Cantor process, as described above,
	\begin{displaymath}
	(\mathbf{s}_1^j,\ldots,\mathbf{s}_n^j) \sim \mathcal{L}[T_j,(p^j_e)_e,k_j], \qquad j = 1,\ldots,m,
	\end{displaymath}
	independently of the other genes.
\end{enumerate}
We define the {\bf multi-locus phylogeny estimation problem} as follows:
\begin{quote}
	Given sequences 
	at the leaves $(\mathbf{s}^j_1,\ldots,\mathbf{s}^j_n)$, $j=1,\ldots,m$, generated by the process above,
	recover the (leaf-labelled) unrooted species tree $S^{-r}$.
\end{quote}
In the context of incomplete lineage sorting,
this problem is the focus of very active
research in statistical phylogenetics~\cite{LYK+:09,DegnanRosenberg:09,ALP:+:12,Nakhleh:13}.
In particular,
there is a number of theoretical results, including~\cite{DegnanRosenberg:06,DeDGiBr+:09,DeGiorgioDegnan:10,MosselRoch:10a,LiYuPe:10,AlDeRh:11,Roch:13,DaNoRo:15,RochSteel:14,RochWarnow:15}. However, many of these results 
concern the statistical properties (identifiability,
consistency, convergence rate) of 
species tree estimators {\em that (unrealistically) assume perfect
knowledge of the $T_j$s.} A very incomplete
picture is available concerning the properties of estimators based
on sequence data, i.e., that 
do {\em not} require the knowledge of the $T_j$s.
(See below for an overview of prior results.)

Here we consider the data requirement of such
estimators based on sequences.
To simplify,
we assume that all genes have the same length, i.e., that $k_j = k$ for all $j=1,\ldots,m$ for some $k$.
(Because our goal is to derive a lower bound,
such simplification is largely immaterial.)
Our results apply to an important class
of methods known as {\em distance-based methods},
which we briefly describe now.
In the single-locus phylogeny estimation
problem, a natural way to infer $T^{-r}$ is
to use the fraction of substitutions between
each pair, i.e., letting $\|\cdot\|_1$ denote the
$\ell_1$-distance,
\begin{equation}\label{eq:theta}
\theta(\mathbf{s}_{a},\mathbf{s}_{b})
:= \|\mathbf{s}_{a} -\mathbf{s}_{b}\|_1,
\qquad \forall a,b \in [n].
\end{equation}
We refer to reconstruction methods relying solely on
the $\theta(\mathbf{s}_{a},\mathbf{s}_{b})$s as distance-based methods.
Assume for instance that $\nu_e = \nu$ for all
$e$, i.e., the so-called molecular clock
hypothesis. Then it is easily seen that single-linkage clustering (e.g.,~\cite{HaTiFr:09}) applied 
to 
the {\em distance matrix}
$(\theta(\mathbf{s}_{a},\mathbf{s}_{b}))_{a,b\in [n]}$
converges to $T^{-r}$ 
as $k \to +\infty$. (In this special case,
the root can be recovered as well.) 
In fact, $T$ can be reconstructed perfectly as long as,
for each $a$, $b$, $\frac{1}{k}\theta(\mathbf{s}_{a},\mathbf{s}_{b})$
is close enough to its expectation (e.g.~\cite{SempleSteel:03}) 
\begin{displaymath}
\theta_{a,b} := k^{-1}\expec[\theta(\mathbf{s}_{a},\mathbf{s}_{b})] = \frac{3}{4}(1 - e^{-d_{ab}})
\quad \text{with} \quad
d_{ab} := \sum_{e \in P(a,b)} \nu_e t_e,
\end{displaymath}
where $P(a,b)$ is the edge set on the unique path
between $a$ and $b$ in $T$. Here ``close enough'' means
$O(f)$ where $f := \min_e \nu_e t_e$. This observation
can been
extended to general $\nu_e$s. 
See e.g.~\cite{ErStSzWa:99a} for explicit bounds
on the sequence length required for perfect reconstruction
with high probability.

Finally, to study distance-based methods in the {\em multi-locus} case, we restrict ourselves to
the following {\bf multi-locus distance estimation problem}:
\begin{quote}
	Given an accuracy $\eps > 0$ and distance matrices $\theta(\mathbf{s}_a^j,\mathbf{s}_b^j)_{a,b\in [n]}$, $j=1,\ldots,m$, estimate $d_{ab}$ as defined above
	within $\eps$ for all $a,b$.
\end{quote}
Observe that,
once the $d_{ab}$s are estimated within sufficient
accuracy, i.e., within $O(f)$,
the species tree can be reconstructed 
using the techniques referred to in the single-locus
case.

\paragraph{Our results}
\begin{figure}
	\centering
	\includegraphics[width = 0.6\textwidth]{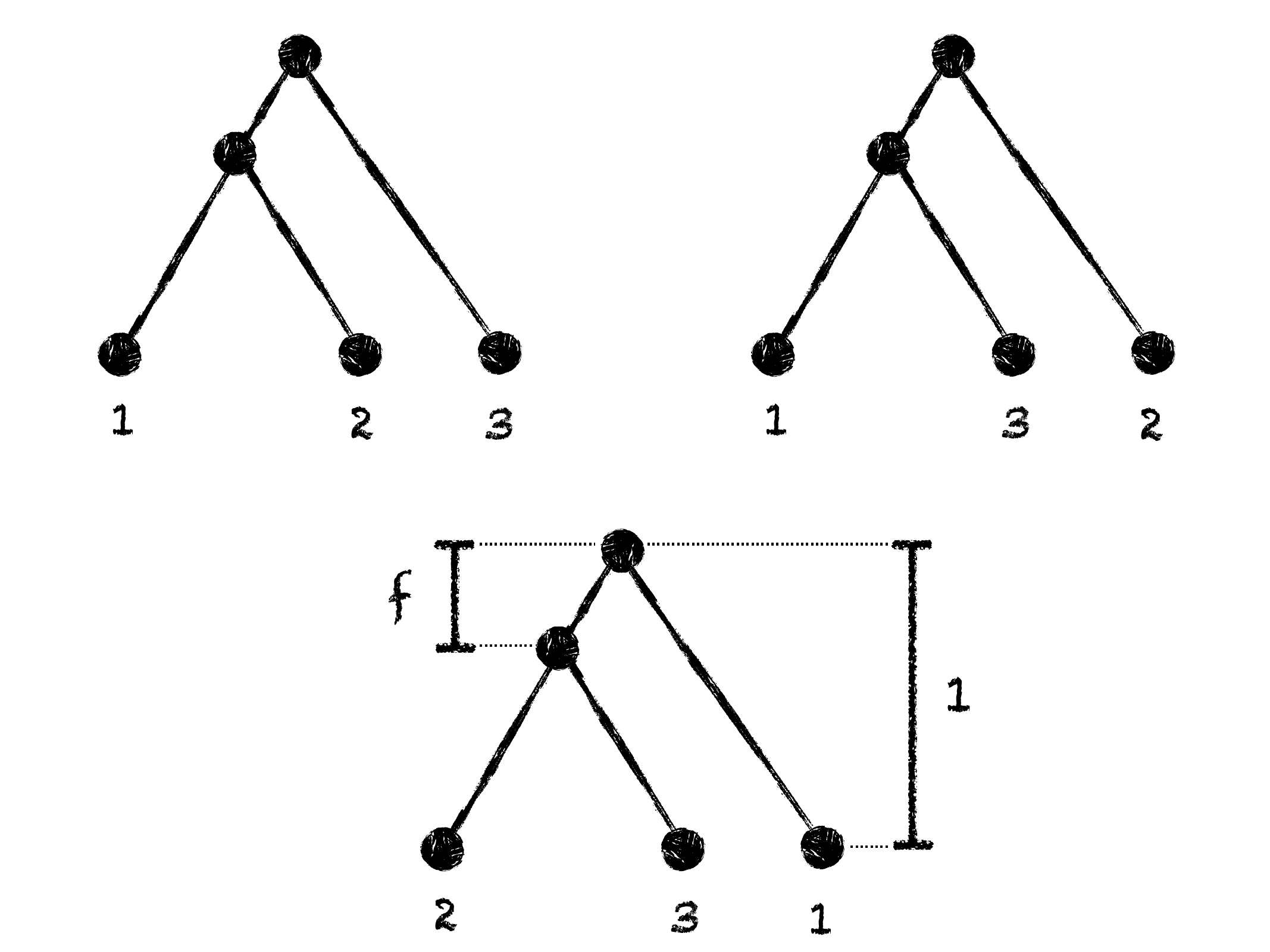}
	\caption{Three species trees.}\label{fig:3-trees}
\end{figure}
How is this related to the sparse signal detection
problem? Our main goal is to provide a lower bound
on the amount of data required for perfect reconstruction,
in terms of $m$ (the number of genes) and $k$ (the sequence length). 
Consider the three possible (rooted, leaf-labelled)
species trees with three leaves, as depicted in 
Figure~\ref{fig:3-trees}, where we let
the time to the most recent divergence be $1 - f$
(from today) and the time to the earlier divergence be $1$. 
Thus $f$ is the time between the two divergence events. 
In order for a distance-based method to distinguish
between these three possibilities, i,e., to determine
which pair is closest,
we need to estimate the $d_{ab}$s within
$O(f)$ accuracy. Put differently, within
the multi-locus distance estimation problem, 
it suffices to establish a lower bound
on the data required
to distinguish between a two-leaf species tree
$S$ with $d_{12} = 2$ and a two-leaf
species tree $S^+$ with $d_{12} = 2 - 2f$,
where in both cases $\nu_e = 1$ for all
$e$. 
We are interested in the limit
$f \to 0$.

Let $\P_0$ and $\Q$ be the distributions of
$\theta(\mathbf{s}^1_1, \mathbf{s}^1_2)$ for a single gene
under $S$
and $S^+$ respectively, where for ease of notation the dependence on
$k$ is implicit. For $m$ genes, we denote the
corresponding distributions by $\P_0^{\otimes m}$
and $\Q^{\otimes m}$.
To connect the problem to sparse signal detection
we observe below that, under the multispecies
coalescent, $\Q$ is in fact a {\em mixture} of
$\P_0$ and a sparse signal $\P_1$, i.e.,
\begin{equation}\label{eq:mixture}
\Q
= (1- \sigma_f)\,\P_0 + \sigma_f\,\P_1,
\end{equation}
where $\sigma_f = O(f)$ as $f \to 0$.

When testing between 
$\P_0^{\otimes m}$ and $\Q^{\otimes m}$, the optimal sum of Type-I (false positive) and Type-II (false negative)
errors is given by (see, e.g.,~\cite{CoverThomas:91})
\begin{equation}\label{eq:testing-error}
\inf_A \{\P_0^{\otimes m}(A) + \Q^{\otimes m}(A^c)\}
= 1 - \|\P_0^{\otimes m} - \Q^{\otimes m}\|_{\mathrm{TV}},
\end{equation}
where $\|\cdot\|_{\mathrm{TV}}$ denotes the total
variation distance. 
Because $\sigma_f = O(f)$, for any $k$,
in order to distinguish between $\P_0$ and $\Q$
one requires that, at the very least, $m = \Omega(f^{-1})$. 
Otherwise the probability of observing
a sample originating from $\P_1$ under $\Q$ is bounded away from $1$. In~\cite{MosselRoch:10a} it was shown
that, provided that $k = \Omega(f^{-2} \log f^{-1})$,
$m = \Omega(f^{-1})$ suffices.
At the other end of the spectrum, when $k = O(1)$,
a lower bound for the single-locus problem obtained by~\cite{SteelSzekely:03}
implies that $m = \Omega(f^{-2})$ is needed. An algorithm achieving this bound under the multispecies coalescent was recently
given in~\cite{DaNoRo:15}.

We settle the full spectrum
between these two regimes.
Our results apply when $k = f^{-2 + 2\kappa}$ and $m = f^{-1 - \mu}$
where $0 < \kappa, \mu < 1$ as $f \to 0$.
\begin{theorem}[Lower bound]\label{thm:main1}
	For any $\delta > 0$,
	there is a $c > 0$ such that
	$$
	\|\P_0^{\otimes m} - \Q^{\otimes m}\|_{\mathrm{TV}}
	\leq \delta,
	$$
	whenever 
	$$
	m \leq c \frac{1}{f^2 \sqrt{k}}.
	$$
\end{theorem}
Notice that the lower bound on $m$ interpolates between the
two extreme regi\-mes discussed above.
As $k$ increases, a more accurate estimate of
the gene trees can be obtained and one expects
that the number of genes required for perfect reconstruction should indeed 
decrease. The form of that dependence is far from
clear however.
We in fact prove that our analysis is tight.
\begin{theorem}[Matching upper bound]\label{thm:main2}
	For any $\delta > 0$,
	there is a $c' > 0$ such that
	$$
	\|\P_0^{\otimes m} - \Q^{\otimes m}\|_{\mathrm{TV}}
	\geq 1 - \delta,
	$$
	whenever 
	$$
	m \geq c' \frac{1}{f^2 \sqrt{k}}.
	$$
	Moreover, there is an efficient
		test to distinguish between $\P_0^{\otimes m}$
		and $\Q^{\otimes m}$ in that case.
\end{theorem}
Our proof of the upper bound actually
gives an efficient reconstruction algorithm under the
molecular clock hypothesis. We expect that the insights
obtained from proving Theorem~\ref{thm:main1}
and~\ref{thm:main2} will lead to more accurate
practical methods as well in the general case.

Our results were announced without proof
in abstract form in~\cite{MosselRoch:u}.

\paragraph{Proof sketch} 
Let $Z$ be an exponential
random variable with mean $1$. We first show that, under $\P_0$
(respectively $\Q$),
$\theta(\mathbf{s}^1_1, \mathbf{s}^1_2)$ is binomial
with $k$ trials and success probability 
$\frac{3}{4}\left(1 - e^{-2(\zeta + Z)}\right)$,
where $\zeta = 1$ (respectively $\zeta = 1 - f$).
Equation~\eqref{eq:mixture} then follows
from the memoryless property of the exponential,
where $\sigma_f$ is the probability that 
$Z \leq f$. 

A recent result of~\cite{CaiWu:14} gives a
formula for the detection boundary of the
sparse signal detection problem for general
$\P_0$, $\P_1$. However, applying this formula
here is non-trivial. Instead we 
bound directly the total variation distance between
$\P_0^{\otimes m}$ and $\Q^{\otimes m}$.
Similarly to the approach used in~\cite{CaiWu:14},
we work with the Hellinger distance $H^2(\P_0^{\otimes m},\Q^{\otimes m})$
which tensorizes as follows (see e.g.~\cite{CoverThomas:91}) 
\begin{equation}\label{eq:tensorize}
\frac{1}{2}H^2(\P_0^{\otimes m},\Q^{\otimes m})
= 1 - \left(
1 - \frac{1}{2}H^2(\P_0,\Q)\right)^m,
\end{equation}
and further satisfies 
\begin{equation}\label{eq:tv-h2}
\|\P_0^{\otimes m} - \Q^{\otimes m}\|^2_{\mathrm{TV}}
\leq H^2(\P_0^{\otimes m},\Q^{\otimes m})
\left[1 - \frac{1}{4}H^2(\P_0^{\otimes m},\Q^{\otimes m})\right].
\end{equation}
All the work is in proving that,
as $f \to 0$,
$$
H^2(\P_0,\Q) = O\left(f^{2} \sqrt{k}\right).
$$
The details are in Section~\ref{section:lower}.

The proof of Theorem~\ref{thm:main2}
on the other hand
involves the construction of a statistical
test that distinguishes between $\P_0^{\otimes m}$
and $\Q^{\otimes m}$. In the regime $k = O(1)$,
an optimal test (up to constants) compares the means of the samples~\cite{DaNoRo:15}. See also~\cite{Liu+:09} 
for a related method (without sample complexity).
In the regime $k = \omega(f^{-2})$, an optimal test (up to constants) compares the minima
of the samples~\cite{MosselRoch:10a}. A natural
way to interpolate between these two tests is 
to consider an appropriate quantile. We show
that a quantile of order $1/\sqrt{k}$ leads to the
optimal choice.

\paragraph{Organization.} 
The gene tree generating model is defined in Section~\ref{section:definitions}. 
The proof of Theorem~\ref{thm:main1}
can be found in Section~\ref{section:lower}.
The proof of Theorem~\ref{thm:main2}
can be found in Section~\ref{section:upper}.

\section{Further definitions}
\label{section:definitions}

In this section, we give more details on
the model.

\paragraph{Some coalescent theory}
\begin{figure}
	\centering
	\includegraphics[width = 0.5\textwidth]{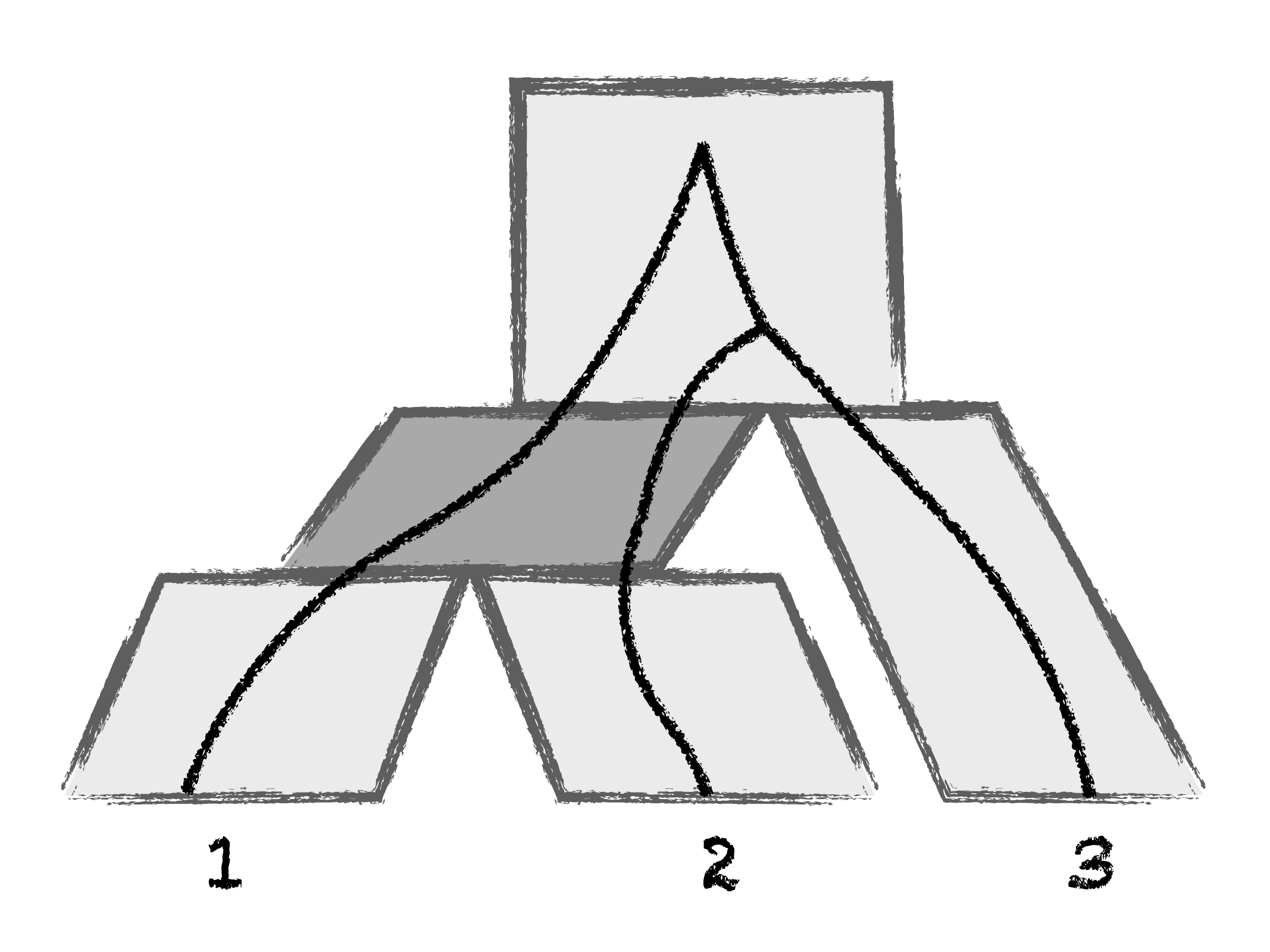}
	\caption{An incomplete lineage sorting event. Although $1$ and $2$ are more closely related in the species tree (fat tree), $2$ and $3$ are more closely related in the gene tree (thin tree). This incongruence is caused by the failure of the lineages originating
		from $1$ and $2$ to coalesce within the shaded branch.}\label{fig:3-species}
\end{figure}
As we mentioned in the previous section,
our gene tree distribution model $\mathcal{G}[S,(\nu_e,t_e)_{e \in E}]$
is the {\em multispecies coalescent}~\cite{RannalaYang:03}.
We first explain the model in the two-species case.
Let $1$ and $2$ be two species and consider a common gene
$j$. One can trace {\em back in time} 
the lineages of gene $j$
from an individual in $1$ and from
an individual in $2$ until the first {\em common}
ancestor. The latter event is called a {\em coalescence}.
Here, because the two lineages originate from different
species, coalescence occurs in an ancestral population.
Let $\tau$ be the time of the divergence between
$1$ and $2$ (back in time). Then, under the 
multispecies coalescent, the {\em coalescence time}
is $\tau + Z$ where $Z$ is an exponential random variable
whose mean depends on the effective population
size of the ancestral population. Here we scale
time so that the mean is $1$. (See, e.g.,~\cite{Durrett:08} for an introduction to coalescent theory.)

We  get for the two-level model of sequence data:
\begin{lemma}[Distance distribution]
Let $S$ be a two-leaf species tree with $d_{12} = 2\tau$
and $\nu_e = 1$ for all $e$ and
let $\theta(\mathbf{s}^1_1,\mathbf{s}^1_2)$ be as
in~\eqref{eq:theta} for some $k$. Then the (random) distribution of
$\theta(\mathbf{s}^1_1,\mathbf{s}^1_2)$ is binomial
with $k$ trials and success probability 
$\frac{3}{4}\left(1 - e^{-2(\tau + Z)}\right)$.
\end{lemma}
The memoryless property of the exponential
gives:
\begin{lemma}[Mixture]\label{lemma:mixture}
Let $S$ be a two-leaf species tree
with $d_{12} = 2$ 
and let $S^+$ be a two-leaf
species tree  with $d_{12} = 2 - 2f$,
where in both cases $\nu_e = 1$ for all
$e$. Let $\P_0$ and $\Q$ be the distributions of
$\theta(\mathbf{s}^1_1, \mathbf{s}^1_2)$ for a single gene
under $S$ and $S^+$ respectively. Then, there is
$\P_1$ such that,
\begin{displaymath}
\Q
= (1- \sigma_f)\,\P_0 + \sigma_f\,\P_1,
\end{displaymath}
where $\sigma_f = O(f)$, as $f \to 0$. 
\end{lemma}
\begin{proof}
The proof of the lemma is straightforward:  We couple perfectly the coalescence time for $\Q$ conditioned on $Z \geq f$ and the unconditional coalescence time for $\P_0$ and this extends to a coupling of the distances between the sequences. Thus $\P_1$ is obtained by conditioning $\Q$ on the event that $Z$ is $\leq f$ and $\sigma_f$ is the probability of that event.
\end{proof}

More generally (this paragraph may be skipped as it will not play a role below), consider a species tree 
$S = (V,E;L,r)$
with 
$n$ leaves. Each gene $j = 1,\ldots, m$ 
has a genealogical history 
represented by its gene tree $T_j$
distributed
according to the following process: 
looking backwards in time, 
on each branch of the species tree, 
the coalescence of any two lineages 
is exponentially distributed with rate 1, 
independently from all other pairs; 
whenever two branches merge in the species tree, 
we also merge the lineages of the corresponding populations, that is, the coalescence proceeds on the \emph{union} of the lineages. 
More specifically, the probability density of a realization of this model for $m$ independent genes is
\begin{align*}
	&\prod_{j=1}^m \prod_{e\in E} 
	\exp\left(-\binom{O_j^{e}}{2} 
	\left[\sigma_j^{e, O_j^{e}+1} - \sigma_j^{e, O_j^{e}}\right]\right)
	\prod_{\ell=1}^{I_j^{e}-O_j^{e}}
	\exp\left(-\binom{\ell}{2} 
	\left[\sigma_j^{e, \ell} - \sigma_j^{e, \ell-1}\right]\right),
\end{align*}
where, for gene $j$ and branch $e$, 
$I_j^{e}$ is the number of lineages entering $e$,
$O_j^{e}$ is the number of lineages exiting $e$, and
$\sigma_j^{e,\ell}$ is the $\ell^{th}$ coalescence time in $e$;
for convenience, we let $\sigma_j^{e,0}$ and
$\sigma_j^{e,I_j^{e}-O_j^{e}+1}$ be respectively
the divergence times of $e$ and of its parent population.
The resulting trees $T_j$s may have topologies that differ
from that of the species tree $S$. This may occur
as a result of an incomplete lineage sorting event,
i.e., the failure of two lineages to coalesce in
a population. See Figure~\ref{fig:3-species} for
an illustration.

\paragraph{A more abstract setting}
Before proving Theorem~\ref{thm:main1}, we re-set
the problem in a more generic setting that will make 
the computations more transparent.  
Let $\P_0$ and $\P_1$ denote two different distributions for a random variable $X$ supported on $[0,1]$. 
Given these distributions,  
we define two distributions, which we will also denote by 
$\P_0$ and $\P_1$, for a random variable $\theta$ 
taking values in $\{0,\ldots,k\}$ for some $k$.
These are defined by 
\begin{equation}\label{eq:theta-distribution}
\P_i[\theta = \ell] = \binom{k}{\ell} \E_i[X^\ell (1 - X)^{k-\ell}],
\end{equation}
where $\E_i$ is the expectation operator corresponding
to $\P_i$ for the random variable $X$ defined on $[0,1]$.
As before, we let
$$
\Q = (1- \sigma_f)\,\P_0 + \sigma_f\,\P_1,
$$
for some $\sigma_f = O(f)$.
We make the following assumptions which
are satisfied in the setting of the previous section.
\begin{enumerate}
	\item[A1.] {\bf Disjoint supports:} $X$ admits a density
	whose support is $(p_0,p^0)$ under $\P_0$ and
	$(p_0 - \phi_f, p_0)$ under $\P_1$, where $0 < p_0 < p^0 < 1$ (independent
	of $f$) and $\phi_f = O(f)$. (In the setting of Lemma~\ref{lemma:mixture}, $p_0 = \frac{3}{4}(1 - e^{-2})$, $p_0 - \phi_f = \frac{3}{4}(1 - e^{-(2-2f)})$, and $p^0 = 3/4$.)
	
	\item[A2.] {\bf Bounded density around $p^0$:} There exist $\rho  \in (0,1)$ and $\bar{p}  \in (p_0, p^0)$, 
	not depending on $f$, such that the following
	holds. Under $\P_0$, the density
	of $X$ on $(p_0,\bar{p})$ lies in the interval $[\rho,\rho^{-1}]$, i.e., for any measurable
	subset $\mathcal{X} \subseteq (p_0,\bar{p})$
	we have
	$$
	\P_0[X \in \mathcal{X}] \in  \left[\rho |\mathcal{X}|, \rho^{-1} |\mathcal{X}|\right],
	$$ 
	where $|\mathcal{X}|$ is the Lebesgue measure of $\mathcal{X}$.
	(In the setting of Lemma~\ref{lemma:mixture}, under $\P_0$
	the density of $X$ on $(p_0,p^0)$ is 
	$\frac{4 e^{1/2}}{3}(1-4x/3)^{-3/4}$. Notice that
	this density is {\bf not} bounded from below over
	the entire interval $(p_0,p^0)$.)
	
\end{enumerate}
The first assumption asserts that the supports of $X$
under $\P_0$ and $\P_1$ are disjoint, while also being highly concentrated under $\P_1$ (as $f \to 0$).
The key point being that, under $\P_1$, $X$ lies
near the lower end of the support under $\P_0$, 
which partly explains the effectiveness of a quantile-based test to distinguish between $\P_0$ and $\Q$.
The second, more technical, assumption asserts that, under $\P_0$, the density of $X$ is bounded from above and below in a neighborhood of the lower end of its support. As
we will see in Section~\ref{section:lower}, 
the dominant contribution to the difference between
$\P_0$ and $\Q$ comes from the regime where
$X$ lies close to $p_0$ and we will need to control
the probability of observing $X$ there.

\section{Lower bound}
\label{section:lower}

The proof of the lower bound is based on establishing an upper bound on the Hellinger distance between $\P_0$ and $\Q$. 
The tensoring property of the Hellinger distance then allows to directly obtain an upper bound on the
Hellinger distance between 
$\P_0^{\otimes m}$ and $\Q^{\otimes m}$. 
Using a standard inequality, this finally gives the desired bound on the total variation distance between $\P_0^{\otimes m}$ and $\Q^{\otimes m}$. 

We first rewrite the Hellinger distance in a form that is convenient for asymptotic expansion. 
In the abstract setting of Section~\ref{section:definitions}, 
the Hellinger distance can be written as
\begin{eqnarray}
H^2(\P_0,\Q)
&=& \sum_{j=0}^k \left[\sqrt{\Q[\theta = j]} - \sqrt{\P_0[\theta=j]}\right]^2\nonumber\\
&=& \sum_{j=0}^k \left[\sqrt{1+\sigma_f\left(\frac{\P_1[\theta=j]}{\P_0[\theta=j]} - 1 \right)}-1\right]^2
\P_0[\theta = j]\nonumber\\
&=& \sum_{j=0}^k \left[\sqrt{1+\sigma_f\left(\frac{ \E_1[X^j (1 - X)^{k-j}]}{ \E_0[X^j (1 - X)^{k-j}]} - 1 \right)}-1\right]^2
\P_0[\theta = j]\nonumber\\
&=& \sum_{j=0}^k h_{\sigma_f}\left(\frac{ \E_1[X^j (1 - X)^{k-j}]}{ \E_0[X^j (1 - X)^{k-j}]}\right)
\P_0[\theta = j], \label{eq:h2totalsum}
\end{eqnarray}
where we define
\begin{equation} \label{eq:defhb}
h_b(s) := (\sqrt{1 + b(s-1)} - 1)^2.
\end{equation} 
We will refer to
$$
\frac{ \E_1[X^j (1 - X)^{k-j}]}{ \E_0[X^j (1 - X)^{k-j}]}
= \frac{ \binom{k}{j}\E_1[X^j (1 - X)^{k-j}]}{ \binom{k}{j}\E_0[X^j (1 - X)^{k-j}]}
= \frac{\P_1[\theta=j]}{\P_0[\theta=j]},
$$
as the {\em likelihood ratio} and
to
$$
\P_0[\theta = j],
$$
as the {\em null probability}.

We prove the following proposition, which
implies Theorem~\ref{thm:main1}.
\begin{proposition}\label{prop:main1}
Assume that
$k = f^{-2 + 2\kappa}$
where $0 < \kappa < 1$
and that Assumptions A1 and A2 hold.
As $f \to 0$,
$$
H^2(\P_0,\Q) = O\left(f^{2} \sqrt{k}\right).
$$
\end{proposition}
\noindent The proof of Proposition~\ref{prop:main1}
follows in the next section.

Finally:
\begin{proof}[Proof of Theorem~\ref{thm:main1}]
The tensorization property of the Hellinger distance, as stated in~(\ref{eq:tensorize}), together with Proposition~\ref{prop:main1} imply that 

\[
\frac{1}{2}H^2(\P_0^{\otimes m},\Q^{\otimes m})
= 1 - \left(
1 - \frac{1}{2}H^2(\P_0,\Q)\right)^m =  1 - \left(
1 - O\left(f^{2} \sqrt{k}\right) \right)^m < \delta, 
\] 
if $m \leq c f^{-2} k^{-1/2}$ for a small enough constant $c$. 
Thus, by (\ref{eq:tv-h2}), we have 
\[
\|\P_0^{\otimes m} - \Q^{\otimes m}\|^2_{\mathrm{TV}}
\leq H^2(\P_0^{\otimes m},\Q^{\otimes m})
\left[1 - \frac{1}{4}H^2(\P_0^{\otimes m},\Q^{\otimes m})\right] < \delta,
\]
as needed. 

\end{proof}

\subsection{Proof of Proposition~\ref{prop:main1}}
\label{section:proof-sketch}

From~\eqref{eq:h2totalsum}, in order to 
bound the Hellinger distance from above, we need
upper bounds on the likelihood ratio $\frac{ \E_1[X^j (1 - X)^{k-j}]}{ \E_0[X^j (1 - X)^{k-j}]}$ and on the
null probability
$\P_0[\theta = j]$ for each term in the sum. 
The basic intuition is that the contributions of those terms where $\theta$ is far from its mean under $\P_1$
(which is $\approx p_0$) are negligible. Indeed: 
\begin{itemize}
	\item When $\theta$ is much smaller than $p_0$,
	the null probability is negligible because, under $\P_0$, $X$ is almost surely greater than $p_0$. We establish that this leads to an
	overall contribution to the Hellinger distance
	of $o(f^2 \sqrt{k})$. See~\eqref{eq:small-thetas}. 
	
	\item When
	$\theta$ is much larger than $p_0$, the likelihood ratio is negligible 
	because $X$ has a much broader support under $\P_0$ than
	it does under $\P_1$.
	In that case, we show that the overall contribution to the
	Hellinger distance is $O(f^2)$. To get a sense of why that is, note that
	as $f \to 0$
	$$
	h_{\sigma_f}(0) = [\sqrt{1 - \sigma_f} - 1]^2 = O(f^2).
	$$
	See Claims~\ref{lemma:less-than-1},~\ref{lemma:JJ1}
	and~\ref{claim:high-substitution}.
	
\end{itemize}
On the other hand, by~\eqref{eq:theta-distribution},
under both $\P_0$ and $\P_1$, the random variable $\theta$ conditioned on $X$ is binomial with mean $kX$ and
standard deviation of order $\sqrt{k}$. 
In the regime considered under Theorem~\ref{thm:main1},
i.e., $k = f^{-2 + 2\kappa}$, we have further that
$f = o(1/\sqrt{k})$.
Hence by Assumption A1, under $\P_1$, $X$ has support 
of size $O(f)$ and the unconditional random variable $\theta$ also has standard deviation of order $\sqrt{k}$. In this bulk regime, our analysis
relies on the following insight:
\begin{itemize}
	\item How big is each term in the Hellinger sum?
	In order for $\E_0[X^j (1 - X)^{k-j}]$ to be non-negligible, $X$ must lie within roughly $\sqrt{k}$ of $p_0$, which under Assumption A2 has probability $\Theta(1/\sqrt{k})$. On the other hand, under $\P_1$, $X$ is almost
	surely close to $p_0$. That produces a likelihood ratio
	of order $\sqrt{k}$. 
	Therefore, recalling that $f \sqrt{k} = o(1)$, the term 
	\[
	h_{\sigma_f} \left( \frac{ \E_1[X^j (1 - X)^{k-j}]}{ \E_0[X^j (1 - X)^{k-j}]} \right) = 
	\left[\sqrt{1+\sigma_f\left(\frac{ \E_1[X^j (1 - X)^{k-j}]}{ \E_0[X^j (1 - X)^{k-j}]} - 1 \right)}-1\right]^2, 
	\]
	is of order $f^2 k$. Moreover, by the argument
	above, the overall null probability of the bulk is of order $1/\sqrt{k}$. Thus, we expect that the Hellinger distance in this regime is of order 
	$\sqrt{k} f^2$ as stated in Proposition~\ref{prop:main1}.
	It will be convenient to divide the analysis
	into $\theta$-values below $p_0$ (see Claims~\ref{claim:low-sub-ratio},~\ref{claim:low-sub-int},~\ref{claim:low-sub-int2} and~\ref{claim:low-substitution}) and
	above $p_0$ (see Claims~\ref{claim:border-ratio},~\ref{claim:border-int},~\ref{claim:border-int2} and~\ref{claim:border-regime}).
\end{itemize}
 The full details are somewhat delicate, as we need to carefully consider various intervals of summands $j$ according to the behavior of the null probability 
 $\P_0[\theta = j]$ and the likelihood ratio $\frac{ \E_1[X^j (1 - X)^{k-j}]}{ \E_0[X^j (1 - X)^{k-j}]}$. 

In the next subsection we introduce some notation and prove some simple estimates that will be used in the proofs.

\subsection{Some useful lemmas}
\label{subsec:useful}

The following is Lemma 4 in~\cite{CaiWu:14}:
\begin{lemma}\label{lemma:h2}
For $b > 0$, let
$
h_b(s) = (\sqrt{1 + b(s-1)} - 1)^2
$ 
\begin{enumerate}
\item For any $b > 0$, the function
$h_b(s)$ is strictly decreasing
on $[0,1]$ and strictly increasing on $[1,+\infty)$.

\item For any $b > 0$ and $s \geq 1$,
$$
h_b(s) 
\leq [b(s-1)] \land [b(s-1)]^2
\leq [bs] \land [bs]^2.
$$
\end{enumerate}
\end{lemma}
The following lemmas follow from straightforward 
calculus. 
\begin{lemma}\label{lemma:entropy}
For $j \in \{0,\ldots,k\}$ and $x \in (0,1)$, let 
$$
\Phi_j(x) = \frac{j}{k}\log x + \frac{k-j}{k}\log (1 - x).
$$
Then
$$
\Phi'_j(x) = \frac{j}{k}\frac{1}{x} - \frac{k-j}{k}\frac{1}{1 - x}
= \frac{1}{x(1 - x)}\left(\frac{j}{k} - x\right).
$$
As a result $\Phi_j$ is increasing on $[0,\frac{j}{k}]$
and decreasing on $[\frac{j}{k},1]$, and
$
\Phi'_j(\frac{j}{k}) = 0. 
$
\end{lemma}
\begin{lemma}\label{lemma:distance}
For $j \in \{0,\ldots,k\}$, $p \in (0,1)$, and $x \in [0,p)$, let 
$$
\Psi_{j,p}(x) = \frac{j}{k}\log \frac{p}{p-x} + \frac{k-j}{k}\log \frac{1-p}{1-p+x}.
$$
Then:
\begin{enumerate}
\item The first two derivatives are:
$$
\Psi'_{j,p}(x) = \frac{j}{k}\frac{1}{p-x} - \frac{k-j}{k} \frac{1}{1-p+x}
= \frac{1}{(p-x)(1-p+x)}\left(\frac{j}{k} - (p-x)\right),
$$
and
$$
\Psi''_{j,p}(x) = \frac{j}{k}\left\{\frac{1}{(p-x)^2}\right\} + \frac{k-j}{k} \left\{\frac{1}{(1-p+x)^2}\right\} \geq \frac{1}{2},
$$
(since the terms in curly brackets are at least $1$
and one of $\frac{j}{k}$ or $\frac{k-j}{k}$ is greater
or equal than $1/2$).

\item By a Taylor expansion around $x = 0$, we have
for $x \in [0,p)$ and some $x^* \in [0,x]$
$$
\Psi_{j,p}(x) 
= \frac{1}{p(1-p)}\left(\frac{j}{k} - p\right) x
+ \frac{x^2}{2} \Psi''_{j,p}(x^*)
\geq \frac{1}{p(1-p)}\left(\frac{j}{k} - p\right) x 
+ \frac{1}{4} x^2.
$$
\end{enumerate}
\end{lemma}


\subsection{Proof} 
 
Let $C$ be a large constant (not depending on $f$)
to be determined later. 
We divide up the sum in~\eqref{eq:h2totalsum} into
intervals with distinct behaviors. We consider the
following intervals for $\frac{j}{k}$:
\begin{equation*}
J_0 = \left[p_0,p_0 + C \sqrt{\frac{\log k}{k}}\right],\qquad
J_1 = \left[p_0 + C \sqrt{\frac{\log k}{k}}, 1\right],
\end{equation*}
and
\begin{eqnarray*}
J'_{0} &=& \left[p_0 - \phi_f, p_0\right],\\
J'_{1} &=& \left[p_0 - C \sqrt{\frac{\log k}{k}}, p_0 - \phi_f\right],\\
J'_{2} &=& \left[0,p_0 - C \sqrt{\frac{\log k}{k}}\right].
\end{eqnarray*}
In words $J_1' \cup J_0' \cup J_0$ is the bulk of $\P_1$, i.e., where $j/k$ sampled from $\P_1$ takes its typical values, with $J_0'$ being
the support of $X$ under $\P_1$. (This bulk interval
is further sub-divided into three intervals whose analyses are slightly different.) The intervals $J_2'$ and $J_1$ are where $j/k$ takes atypically small and large values under $\P_1$ respectively. 
For a subset of $\frac{j}{k}$-values $J$, we write 
the contribution of $J$ to the Hellinger distance as
$$
H^2(\P_0,\Q)|_J
= 
\sum_{j: j/k \in J} h_{\sigma_f} \left(\frac{ \E_1[X^j (1 - X)^{k-j}]}{ \E_0[X^j (1 - X)^{k-j}]}\right)
\P_0[\theta = j].
$$
Below, it will be convenient to break up the analysis into three regimes:
$J_1$, which we refer to as the {\em high-substitution
	regime}; 
$J_2' \cup J_1' \cup J_0'$, the {\em low-substitution regime};
and $J_0$, the {\em border regime}. (Refer back to Section~\ref{section:proof-sketch}
for an overview of the proof in these different regimes.
Note in particular that we combine the analyses
of the atypically low values, $J_2'$, and the
typical values below $p_0$, $J_1' \cup J_0'$, because
they follow from related derivations.)

\paragraph{High substitution regime}
We consider $J_1$ first. As we previewed in
Section~\ref{section:proof-sketch}, the argument
in this case involves proving that the likelihood
ratio is small.
Let
\begin{equation} \label{eq:defJle1}
J_{\leq 1} = \left\{
0 \leq j \leq k\,:\,\frac{ \E_1[X^j (1 - X)^{k-j}]}{ \E_0[X^j (1 - X)^{k-j}]} \leq 1
\right\}.
\end{equation}
I.e., $J_{\leq 1}$ is where the likelihood ratio
is bounded by $1$. 
Note that 
Lemma~\ref{lemma:h2} in Section~\ref{subsec:useful} says that $h_{\sigma_f}$ is monotone decreasing in the interval $[0,1]$ and we can therefore bound the sum of terms in $J_{\leq 1}$ assuming the likelihood ratio 
$\frac{ \E_1[X^j (1 - X)^{k-j}]}{ \E_0[X^j (1 - X)^{k-j}]}$ in fact equals $0$, as follows,
\begin{eqnarray*}
	&&\sum_{j\in J_{\leq 1}} \left[\sqrt{1+\sigma_f\left(\frac{ \E_1[X^j (1 - X)^{k-j}]}{ \E_0[X^j (1 - X)^{k-j}]} - 1 \right)}-1\right]^2
	\P_0[J = j]\\
	&&\qquad \leq \sum_{j\in J_{\leq 1}} \left[\sqrt{1-\sigma_f}-1\right]^2
	\P_0[J = j]\\
	&&\qquad = O(\sigma_f^2)\\ 
	&&\qquad = O(f^2). 
\end{eqnarray*}
We have thus proved the following claim.
\begin{claim}[Ratio less than $1$]\label{lemma:less-than-1}
	$$
	H^2(\P_0,\Q)|_{J_{\leq 1}} = O(f^2).
	$$
\end{claim}
\noindent Hence, to bound the sum in $J_1$, it suffices to show that 
$J_1 \subseteq J_{\leq 1}$, which we prove in the next claim. 
\begin{claim}[High substitution implies ratio less than 1] \label{lemma:JJ1}
	It holds that
$J_1 \subseteq J_{\leq 1}$. 
\end{claim}
\noindent Since the support of $X$ under $\P_1$
is below $p_0$ while it is above $p_0$ 
under $\P_0$, we might expect that 
the likelihood ratio will be bounded by $1$ on $J_1$, 
which is what we prove next. 
\begin{proof}
By Assumption A1, under $\P_1$, $X$ is a.s.~less than $p_0$. Since Lemma~\ref{lemma:entropy} implies that $\Phi_j(x)$ is monotone increasing on $[0,j/k]$,
which includes $[0,p_0]$ since $\frac{j}{k} \in J_1$, it follows that 
\begin{eqnarray}
\E_1[X^j (1 - X)^{k-j}] 
&=& \E_1[\exp(k\Phi_{j}(X))]\nonumber\\
&\leq& \exp(k \Phi_j(p_0))\nonumber\\
&=& p_0^j(1-p_0)^{k-j}.\label{eq:j1e1}
\end{eqnarray}
Let $\ecal$ be the event that 
$$
\ecal = \left\{X \in \left[p_0 + C\sqrt{\frac{\log k}{k}} - \frac{1}{k},
p_0 + C\sqrt{\frac{\log k}{k}}\right]\right\}.
$$
By Assumption A2, $\P_0[\ecal] \geq \rho/k$.
Hence, using Lemma~\ref{lemma:entropy} again,
for $\frac{j}{k} \in J_1$
\begin{eqnarray}
\E_0[X^j (1 - X)^{k-j}]
&=&   \E_0[X^j (1 - X)^{k-j}\,|\,\ecal] \P_0[\ecal]
+ \E_0[X^j (1 - X)^{k-j}\,|\,\ecal^c] \P_0[\ecal^c]\nonumber\\
&\geq& \frac{\rho}{k} p^j(1-p)^{k-j},\label{eq:j1e0}
\end{eqnarray}
where $p = p_0 + C\sqrt{\frac{\log k}{k}} - \frac{1}{k}$. 

Combining~\eqref{eq:j1e1} and~\eqref{eq:j1e0}, 
and using Lemma~\ref{lemma:distance}
with $x = p - p_0 \geq 0$ (for $k > 1$ and $C$ large enough),
we have
\begin{eqnarray*}
	\frac{\E_0[X^j (1 - X)^{k-j}]}{\E_1[X^j (1 - X)^{k-j}]}
	&\geq& \frac{\rho p^j(1-p)^{k-j}}{k p_0^j(1-p_0)^{k-j}}\\
	&=& \frac{\rho}{k} \exp\left(
	k \Psi_{j,p}(p - p_0)
	\right)\\
	&\geq& \frac{\rho}{k} \exp\left(
	k \left\{\frac{1}{p(1-p)}\left(\frac{j}{k} - p\right) x 
	+ \frac{1}{4} x^2\right\}
	\right)\\
	&\geq& \frac{\rho}{k} \exp\left(
	\frac{k}{4} 
	\left(
	C\sqrt{\frac{\log k}{k}} - \frac{1}{k}
	\right)^2
	\right)\\
	&\geq& \frac{\rho}{k} \exp\left(
	\frac{C^2}{5} 
	\log k
	\right)\\
	&\geq& 1,
\end{eqnarray*}
for $C$ large enough (assuming $k$ is large),
where on the fourth line we used that
$j/k-p \geq 0$ for $j/k \in J_1 = \left[p_0 + C \sqrt{\frac{\log k}{k}}, 1\right]$.
We have thus established $J_1 \subseteq J_{\leq 1}$
\end{proof}

Combining Claims~\ref{lemma:less-than-1} and~\ref{lemma:JJ1}, 
we thus obtain:
\begin{claim}[High substitution: Hellinger distance]\label{claim:high-substitution}
	$$
	H^2(\P_0,\Q)|_{J_1} = O(f^2).
	$$
\end{claim}

\paragraph{Low substitution regime}
In order to estimate the sum in $J'_0\cup J'_{1}\cup J'_{2}$ we need to further subdivide it into intervals of doubling length. 
The basic intuition is that for far enough intervals the null probabilities $\P_0[\theta = j]$  are small enough so we can estimate the likelihood ratio term 
$\frac{ \E_1[X^j (1 - X)^{k-j}]}{ \E_0[X^j (1 - X)^{k-j}]}$ by its worst value in the interval. However, when the intervals are close to the mean, the fluctuations in 
$\frac{ \E_1[X^j (1 - X)^{k-j}]}{ \E_0[X^j (1 - X)^{k-j}]}$ are too big so we need to work with shorter intervals. The partition is defined as follows:

\begin{eqnarray*}
	I'_0 &=& \left[p_0 - \frac{1}{\sqrt{k}}, p_0\right]\\
	I'_\ell &=& \left[p_0 - \frac{2^{\ell}}{\sqrt{k}}, p_0 - \frac{2^{\ell-1}}{\sqrt{k}}\right],\qquad \ell \geq 1.
\end{eqnarray*}
Define $L$ by $2^L = C\sqrt{\log k}$ (where we may choose $C$ so 
that it is integer-valued).

We first upper bound $\E_1[X^j (1 - X)^{k-j}]$ using
Lemma~\ref{lemma:entropy}
and Assumption A1:
\begin{itemize}
	\item On $J'_0$,
	\begin{eqnarray}
	\E_1[X^j (1 - X)^{k-j}] 
	&=& \E_1[\exp(k \Phi_j(X))]\nonumber\\ 
	&\leq& \E_1[\exp(k \Phi_j(j/k))]\nonumber\\ 
	&=& (j/k)^j(1-j/k)^{k-j}.\label{eq:jp0e1}
	\end{eqnarray}
	
	\item On $J'_1 \cup J'_2$ we have that $X \geq p_0 - \phi_f$ a.s. and therefore  
	\begin{equation}\label{eq:jp12e1}
	\E_1[X^j (1 - X)^{k-j}] \leq (p_0-\phi_f)^j(1-p_0+\phi_f)^{k-j}.
	\end{equation}
	
\end{itemize}
To lower bound $\E_0[X^j (1 - X)^{k-j}]$,
we consider the event
$$
\ecal = \left\{X \in \left[p_0,
p_0 + \sqrt{\frac{1}{k}}\right]\right\}.
$$
By Assumption A2 and Lemma~\ref{lemma:entropy},
on $J'_0 \cup J'_1 \cup J'_2$, arguing as in~\eqref{eq:j1e0},
\begin{eqnarray}
\E_0[X^j (1 - X)^{k-j}]
&\geq& \frac{\rho}{\sqrt{k}} p^j(1-p)^{k-j},\label{eq:jp012e0}
\end{eqnarray}
where $p = p_0 + \sqrt{\frac{1}{k}}$ (assuming $k$ is large).
Combining~\eqref{eq:jp0e1},~\eqref{eq:jp12e1},
and~\eqref{eq:jp012e0}, and using Lemma~\ref{lemma:distance}:
\begin{itemize}
	\item On $J'_0$,{\small
	\begin{eqnarray*}
		\frac{\E_0[X^j (1 - X)^{k-j}]}{\E_1[X^j (1 - X)^{k-j}]} 
		&\geq& \frac{\rho}{\sqrt{k}}
		\exp(k \Psi_{j,p}(p-j/k))\\
		&\geq& \frac{\rho}{\sqrt{k}}
		\exp\left(k\left(-\frac{1}{p(1-p)}\left(\frac{j}{k} - p\right)^2 
		+ \frac{1}{4} \left(\frac{j}{k} - p\right)^2\right)\right)\\
		&\geq& C'_1 \frac{\rho}{\sqrt{k}},
	\end{eqnarray*}}
	\!for some constant $C'_1$ (not depending on $f$), 
	where we used that $\phi_f \ll \sqrt{1/k}$ so that $\left(\frac{j}{k} - p\right)^2 = O(1/k)$ and, further,
	$p(1-p)\in (0,1/4)$.
	
	\item On $J'_1 \cup J'_2$,{\small
	\begin{eqnarray*}
		&&\frac{\E_0[X^j (1 - X)^{k-j}]}{\E_1[X^j (1 - X)^{k-j}]}\\ 
		&&\quad \geq \frac{\rho}{\sqrt{k}}
		\exp(k \Psi_{j,p}(p-p_0+\phi_f))\\
		&&\quad \geq \frac{\rho}{\sqrt{k}}
		\exp\left(k\left(\frac{1}{p(1-p)}\left(\frac{j}{k} - p\right)(p-p_0+\phi_f)
		+ \frac{1}{4} (p-p_0+\phi_f)^2\right)\right)\\
		&&\quad \geq \frac{\rho}{\sqrt{k}}
		\exp\left(-C_1 \sqrt{k}\left(p - \frac{j}{k}\right)\right)\\
		&&\quad = \frac{\rho}{\sqrt{k}}
		\exp\left(-C_1 \sqrt{k}\left(p_0 + \sqrt{\frac{1}{k}} - \frac{j}{k}\right)\right)\\		
		&&\quad = C_2 \frac{\rho}{\sqrt{k}}
		\exp\left(-C_1 \sqrt{k}\left(p_0-\frac{j}{k}\right)\right),
	\end{eqnarray*}}
	\!for some constants $C_1, C_2$ (not depending on $f$), 
	where again we used that $\phi_f \ll \sqrt{1/k}$
	so that $(p-p_0+\phi_f)^2 = O(1/k^2)$.
	
\end{itemize} 
By decreasing $C_2$
appropriately we combine the two bounds into:
\begin{claim}[Low substitution: Likelihood ratio]
	\label{claim:low-sub-ratio}
	For all $j/k \in J'_0 \cup J'_1 \cup J'_2$,
	\begin{equation}\label{eq:jpratio}
	\frac{\E_1[X^j (1 - X)^{k-j}]}{\E_0[X^j (1 - X)^{k-j}]} 
	\leq  \frac{\sqrt{k}}{C_2}
	\exp\left(C_1 \sqrt{k}\left(p_0-\frac{j}{k}\right)\right).
	\end{equation}
\end{claim}

We now bound the integrand in $H^2(\P_0,\Q)$
over $J'_0 \cup J'_1 \cup J'_2$.
As noted after the definition of $J_{\leq 1}$ in equation~(\ref{eq:defJle1}), Lemma~\ref{lemma:h2}  
implies that on $J_{\leq 1}$
\begin{equation}\label{eq:on-j-less-than-1}
h_{\sigma_f}\left(
\frac{\E_1[X^j (1 - X)^{k-j}]}{\E_0[X^j (1 - X)^{k-j}]}
\right) \leq C_0 f^2,
\end{equation}
for some constant $C_0 > 0$.
\begin{itemize}
	\item On $J'_0 \cup J'_1$, 
	we will further use Lemma~\ref{lemma:h2} (Part 2) which, recall, says that for
	$s \geq 1$ and $b > 0$
	$$
	h_b(s) 
	\leq [b(s-1)] \land [b(s-1)]^2
	\leq [bs] \land [bs]^2.
	$$
	In particular observe that, if $s \geq 1$, $b > 0$ and $bs < 1$, then 
	we have simply $h_b(s) 
	\leq [bs]^2$.
	Here $b = \sigma_f$ and $s$ is bounded 
	above by the expression in~\eqref{eq:jpratio}.
	We show first that $bs$ is therefore small. Indeed,
	$$
	\sigma_f \frac{\sqrt{k}}{C_2}
	\exp\left(C_1 \sqrt{k}\left(p_0-\frac{j}{k}\right)\right)
	= O(f^\kappa) \exp\left(O(\sqrt{\log f^{-1}})\right)
	= o(1).
	$$
	Hence, for those $j/k$-values where the likelihood
	ratio is bounded below by $1$, we have by Lemma~\ref{lemma:h2} (Part 2) that
	$$
	h_{\sigma_f}\left(
	\frac{\E_1[X^j (1 - X)^{k-j}]}{\E_0[X^j (1 - X)^{k-j}]}
	\right) \leq 
	\frac{\sigma^2_f  k}{C_2^2}
	\exp\left(2C_1 \sqrt{k}\left(p_0-\frac{j}{k}\right)\right).
	$$
	For those $j/k$-values where the likelihood
	ratio is {\bf not} bounded below by $1$, we instead use~\eqref{eq:on-j-less-than-1}.
	Changing the constants we obtain finally the following
	bound valid on all of on $J'_0 \cup J'_1$:
	\begin{equation}\label{eq:using-lemma32}
	h_{\sigma_f}\left(
	\frac{\E_1[X^j (1 - X)^{k-j}]}{\E_0[X^j (1 - X)^{k-j}]}
	\right) \leq 
	C_2 f^2  k
	\exp\left(C_1 \sqrt{k}\left(p_0-\frac{j}{k}\right)\right).
	\end{equation}
	
	\item On $J'_2$, arguing as in the previous case, we note that
	the likelihood ratio multiplied by $\sigma_f$ may
	be larger than $1$ this time. Therefore by Lemma~\ref{lemma:h2} (Part 2) and~\eqref{eq:on-j-less-than-1}
	we have{\footnotesize
	\begin{eqnarray*}
		&&h_{\sigma_f}\left(
		\frac{\E_1[X^j (1 - X)^{k-j}]}{\E_0[X^j (1 - X)^{k-j}]}\right)\\
		&&\ \ \leq C_0 f^2\lor \left\{
		\left[\sigma_f \frac{\sqrt{k}}{C_2}
		\exp\left(C_1 \sqrt{k}\left(p_0-\frac{j}{k}\right)\right)\right]
		\land 
		\left[\sigma_f \frac{\sqrt{k}}{C_2}
		\exp\left(C_1 \sqrt{k}\left(p_0-\frac{j}{k}\right)\right)\right]^2\right\},\qquad\qquad
	\end{eqnarray*}}
	\!Changing the constants we re-write this expression 
	as
	$$
	h_{\sigma_f}\left(
	\frac{\E_1[X^j (1 - X)^{k-j}]}{\E_0[X^j (1 - X)^{k-j}]}
	\right) \leq 
	C_2 f  \sqrt{k}
	\exp\left(C_1 \sqrt{k}\left(p_0-\frac{j}{k}\right)\right),
	$$
	where, to upper bound the minimum in square brackets above, we only squared the exponential (which is larger than $1$) and used the fact that
	$f\sqrt{k} = o(1)$ (which implies that the term
	$\sigma_f \frac{\sqrt{k}}{C_2}$ is on the other hand asymptotically
	smaller than $1$). We also used that $f^2 < f \sqrt{k}$ to deal with the maximum above.

\end{itemize}
We combine the two bounds into:
\begin{claim}[Low substitution: Integrand]
	\label{claim:low-sub-int}
	For all $j/k \in J'_0 \cup J'_1 \cup J'_2$,{\small
	\begin{equation}\label{eq:jpintegrand}
	h_{\sigma_f}\left(
	\frac{\E_1[X^j (1 - X)^{k-j}]}{\E_0[X^j (1 - X)^{k-j}]}
	\right) \leq 
	C_2 (f^2  k \ind_{j/k \in J'_0 \cup J'_1} + f  \sqrt{k} \ind_{j/k \in J'_2})
	\exp\left(C_1 \sqrt{k}\left(p_0-\frac{j}{k}\right)\right).
	\end{equation}}
\end{claim}

It remains to bound the integrator, for which
we rely on Chernoff's bound. 
We let
\begin{eqnarray*}
	I_0 &=& \left[p_0, p_0 + \frac{1}{\sqrt{k}}\right]\\
	I_\ell &=& \left[p_0 + \frac{2^{\ell-1}}{\sqrt{k}}, p_0 + \frac{2^{\ell}}{\sqrt{k}}\right],\qquad \ell \geq 1.
\end{eqnarray*}
Let $\Lambda > 0$ be such that $2^\Lambda = (\bar{p}-p_0)\sqrt{k}$.
Then by Assumption A2
\begin{eqnarray*}
	\P_0[\theta/k \in I'_\ell] 
	&=& \sum_{\lambda \geq 0} \P_0[\theta/k \in I'_\ell\,|\,X \in I_\lambda]\,
	\P_0[X \in I_\lambda]\\
	&\leq& \sum_{\lambda \geq 0} \P_0[\theta/k \in I'_\ell\,|\,X \in I_\lambda]
	\frac{2^{\lambda-1}}{\sqrt{k}}\rho^{-1}
	+ \sum_{\lambda > \Lambda} \P_0[\theta/k \in I'_\ell\,|\,X \in I_\lambda].\end{eqnarray*}
By Chernoff's bound
$$
\P_0[\theta/k \in I'_\ell\,|\,X \in I_\lambda]
\leq \exp\left(
-2(2^{\ell-1} + 2^{\lambda-1})^2
\right) 
\leq 
\exp\left(
- 2^{2\ell-1} - 2^{2\lambda-1}
\right).
$$
In particular
\begin{eqnarray*}
	\sum_{\lambda > \Lambda} \P_0[\theta/k \in I'_\ell\,|\,X \in I_\lambda]
	&\leq& \exp\left(
	- 2^{2\ell-1}
	\right) \sum_{\lambda > \Lambda} \exp\left(
	- 2^{2\lambda-1}
	\right)\\
	&\leq& \exp\left(
	- 2^{2\ell-1}
	\right) \exp\left(
	- C'_3 k
	\right),
\end{eqnarray*}
for some constant $C'_3 > 0$ (not depending on $f$).
On the other hand,
\begin{eqnarray*}
	\sum_{\lambda \geq 0} \P_0[\theta/k \in I'_\ell\,|\,X \in I_\lambda]
	\frac{2^{\lambda-1}}{\sqrt{k}}\rho^{-1}
	&\leq& 
	\frac{\exp\left(
		- 2^{2\ell-1}
		\right)}{\rho \sqrt{k}} 
	\sum_{\lambda \geq 0} 2^{\lambda-1}
	\exp\left(
	- 2^{2\lambda-1}
	\right)\\
	&\leq& 
	\frac{C_3 \exp\left(
		- 2^{2\ell-1}
		\right)}{\sqrt{k}},
\end{eqnarray*}
for a constant $C_3 > 0$ (not depending on $f$).
Combining the bounds and increasing $C_3$
appropriately, we get
\begin{claim}[Low substitution: Integrator]
	\label{claim:low-sub-int2}
	For all $\ell \geq 0$,
	\begin{equation}\label{eq:jpintegrator}
	\P_0[\theta/k \in I'_\ell] 
	\leq \frac{C_3 \exp\left(
		- 2^{2\ell-1}
		\right)}{\sqrt{k}}.
	\end{equation}
\end{claim}

We can now compute the contribution of $J'_0 \cup J'_1 \cup J'_2$ to the Hellinger distance.
Recall that $L$ is defined by $2^L = C\sqrt{\log k}$.
From~\eqref{eq:jpintegrand} and~\eqref{eq:jpintegrator},
we get:
\begin{itemize}
	\item For $0 \leq \ell \leq L$,
	\begin{eqnarray*}
		H^2(\P_0,\Q)|_{I'_\ell}
		&\leq& 
		C_2 f^2  k
		\exp\left(C_1 \sqrt{k}\left(\frac{2^\ell}{\sqrt{k}}\right)\right)
		\frac{C_3 \exp\left(
			- 2^{2\ell-1}
			\right)}{\sqrt{k}}\\
		&\leq& 
		C_2 C_3 f^2  \sqrt{k}
		\exp\left(- 2^{2\ell-1} + C_1 2^\ell \right)\\
		&\leq& 
		C_5 f^2  \sqrt{k}
		\exp\left(- C_4 2^{2\ell} \right),
	\end{eqnarray*}
	for some constants $C_4, C_5 > 0$.
	Summing over $\ell$ we get
	$$
	\sum_{\ell = 0}^{L} 
	H^2(\P_0,\Q)|_{I'_\ell}
	\leq C_6 f^2 \sqrt{k},
	$$
	for some constant $C_6 > 0$.
	
	\item Similarly, for $\ell > L$,
	\begin{eqnarray*}
		H^2(\P_0,\Q)|_{I'_\ell}
		&\leq& 
		C_5 f
		\exp\left(- C_4 2^{2\ell} \right),
	\end{eqnarray*}
	adapting constants $C_4, C_5 > 0$.
	Summing over $\ell$ we get
	\begin{equation}\label{eq:small-thetas}
	\sum_{\ell > L} 
	H^2(\P_0,\Q)|_{I'_\ell}
	\leq C_8 f \exp\left( - C_7 C^2 \log k\right)
	= o(f^{1+\kappa}) = o(f^2 \sqrt{k}),
	\end{equation}
	by choosing $C$ large enough.
	
\end{itemize}
Combining these bounds we get finally:
\begin{claim}[Low substitution: Hellinger distance]\label{claim:low-substitution}
	$$
	H^2(\P_0,\Q)|_{J'_0 \cup J'_1 \cup J'_2} = O(f^2 \sqrt{k}).
	$$
\end{claim}

\paragraph{Border regime.}
We now consider $J_0$, i.e., the bulk regime above $p_0$.
The high-level argument is similar to the
case of $J_0'\cup J_1'$ above, although some details
differ.
We first bound $\E_1[X^j (1 - X)^{k-j}]$
using Lemma~\ref{lemma:entropy}
and Assumption A1
\begin{equation}\label{eq:j0e1}
\E_1[X^j (1 - X)^{k-j}] \leq p_0^j(1-p_0)^{k-j}.
\end{equation}
To bound $\E_0[X^j (1 - X)^{k-j}]$,
we consider the event
$$
\ecal = \left\{X \in \left[\frac{j}{k},
\frac{j}{k} + \sqrt{\frac{1}{k}}\right]\right\}.
$$
By Assumption A2 and Lemma~\ref{lemma:entropy},
on $J_0$, arguing as in~\eqref{eq:j1e0},
\begin{eqnarray}
\E_0[X^j (1 - X)^{k-j}]
&\geq& \frac{\rho}{\sqrt{k}} p^j(1-p)^{k-j},\label{eq:j0e0}
\end{eqnarray}
where $p = \frac{j}{k} + \sqrt{\frac{1}{k}}$.
Combining~\eqref{eq:j0e1}
and~\eqref{eq:j0e0}, and using Lemma~\ref{lemma:distance},
on $J_0$,{\small
\begin{eqnarray*}
	\frac{\E_0[X^j (1 - X)^{k-j}]}{\E_1[X^j (1 - X)^{k-j}]} 
	&\geq& \frac{\rho}{\sqrt{k}}
	\exp(k \Psi_{j,p}(p-p_0))\\
	&\geq& \frac{\rho}{\sqrt{k}}
	\exp\left(k\left(\frac{1}{p(1-p)}\left(\frac{j}{k} - p\right)(p - p_0) 
	+ \frac{1}{4} (p-p_0)^2\right)\right).
\end{eqnarray*}}
\!For $j/k \in I_\ell$, $\ell \leq L$,{\small
\begin{eqnarray*}
	\frac{\E_0[X^j (1 - X)^{k-j}]}{\E_1[X^j (1 - X)^{k-j}]} 
	&\geq& \frac{\rho}{\sqrt{k}}
	\exp\left(k\left(-\frac{1}{p(1-p)}\sqrt{\frac{1}{k}}\left(\frac{2^\ell + 1}{\sqrt{k}}\right) 
	+ \frac{1}{4} \left(\frac{2^{\ell-1} + 1}{\sqrt{k}}\right)^2\right)\right)\\
	&\geq& C_2 \frac{1}{\sqrt{k}}
	\exp\left(C_1 2^{2\ell}\right),
\end{eqnarray*}}
\!for constants $C_1, C_2 > 0$.
\begin{claim}[Border regime: Likelihood ratio]
	\label{claim:border-ratio}
	For all $j/k \in I_\ell$, $\ell \leq L$,
	\begin{equation}\label{eq:jratio}
	\frac{\E_1[X^j (1 - X)^{k-j}]}{\E_0[X^j (1 - X)^{k-j}]} 
	\leq  \frac{\sqrt{k}}{C_2}
	\exp\left(- C_1 2^{2\ell}\right).
	\end{equation}
\end{claim}

We now bound the integrand in $H^2(\P_0,\Q)$.
We follow the argument leading up to~\eqref{eq:using-lemma32}.
Because
$
\sigma_f \sqrt{k} = o(1),
$
by Lemma~\ref{lemma:h2} (Part 2) and~\eqref{eq:on-j-less-than-1}
we have on $I_\ell$
$$
h_{\sigma_f}\left(
\frac{\E_1[X^j (1 - X)^{k-j}]}{\E_0[X^j (1 - X)^{k-j}]}
\right) \leq C_0 f^2\lor 
\frac{\sigma^2_f  k}{C_2^2}
\exp\left(-2C_1 2^{2\ell}\right).
$$
Changing the constants we re-write this expression 
as
$$
h_{\sigma_f}\left(
\frac{\E_1[X^j (1 - X)^{k-j}]}{\E_0[X^j (1 - X)^{k-j}]}
\right) \leq C_0 f^2\lor 
C_2 f^2  k
\exp\left(- C_1 2^{2\ell}\right).
$$
\begin{claim}[Border regime: Integrand]
	\label{claim:border-int}
	For all $j/k \in I_\ell$, $0 \leq \ell \leq L$,
	\begin{equation}\label{eq:jintegrand}
	h_{\sigma_f}\left(
	\frac{\E_1[X^j (1 - X)^{k-j}]}{\E_0[X^j (1 - X)^{k-j}]}
	\right) \leq C_0 f^2\lor 
	C_2 f^2  k
	\exp\left(- C_1 2^{2\ell}\right).
	\end{equation}
\end{claim}

It remains to bound the integrator. 
We have by Assumption A2 (recall that $2^\Lambda = (\bar{p}-p_0)\sqrt{k}$)
\begin{eqnarray*}
	\P_0[\theta/k \in I_\ell] 
	&=& \sum_{\lambda \geq 0} \P_0[\theta/k \in I_\ell\,|\,X \in I_\lambda]
	\P_0[X \in I_\lambda]\\
	&\leq& \sum_{0 \leq \lambda \leq \ell} \P_0[X \in I_\lambda] + \sum_{\ell < \lambda \leq \Lambda} \P_0[\theta/k \in I_\ell\,|\,X \in I_\lambda]\P_0[X \in I_\lambda]\\
	&& \qquad+ \sum_{\lambda > \Lambda} \P_0[\theta/k \in I_\ell\,|\,X \in I_\lambda]\\
	&\leq& \frac{2^\ell}{\sqrt{k}}\rho^{-1} + \sum_{\ell < \lambda \leq \Lambda} \P_0[\theta/k \in I_\ell\,|\,X \in I_\lambda]
	\frac{2^{\lambda-1}}{\sqrt{k}}\rho^{-1}\\
	&& \qquad+ \sum_{\lambda > \Lambda} \P_0[\theta/k \in I_\ell\,|\,X \in I_\lambda].
\end{eqnarray*}
By Chernoff's bound, for $\lambda > \ell$,
$$
\P_0[\theta/k \in I_\ell\,|\,X \in I_\lambda]
\leq \exp\left(
-2(-2^{\ell} + 2^{\lambda-1})^2
\right) 
\leq 
\exp\left(-2^{2\ell + 1}(2^{\lambda - \ell - 1} - 1)^2
\right).
$$
In particular
\begin{eqnarray*}
	\sum_{\lambda > \Lambda} \P_0[\theta/k \in I_\ell\,|\,X \in I_\lambda] 
	&\leq& \exp\left(
	- C'_3 k
	\right),
\end{eqnarray*}
for some constant $C'_3 > 0$ (not depending on $f$).
On the other hand,
\begin{eqnarray*}
	\sum_{\ell < \lambda \leq \Lambda} \P_0[\theta/k \in I_\ell\,|\,X \in I_\lambda]
	\frac{2^{\lambda-1}}{\sqrt{k}}\rho^{-1}
	&\leq& 
	\frac{C_3 2^{\ell}}{\sqrt{k}},
\end{eqnarray*}
for a constant $C_3 > 0$ (not depending on $f$).
Combining the bounds and increasing $C_3$
appropriately, we get
\begin{claim}[Border substitution: Integrator]
	\label{claim:border-int2}
	For all $0 \leq \ell < L$,
	\begin{equation}\label{eq:jintegrator}
	\P_0[\theta/k \in I_\ell] 
	\leq \frac{C_3 2^{\ell}}{\sqrt{k}}.
	\end{equation}
\end{claim}

We can now compute the contribution of $J_0$ to the Hellinger distance.
From~\eqref{eq:jintegrand} and~\eqref{eq:jintegrator},
we get for $0 \leq \ell \leq L$
\begin{eqnarray*}
	H^2(\P_0,\Q)|_{I_\ell}
	&\leq& \left[C_0 f^2\lor 
	C_2 f^2  k
	\exp\left(-C_1 2^{2\ell}\right)\right]
	\frac{C_3 2^{\ell}}{\sqrt{k}}.
\end{eqnarray*}
Summing over $\ell$ we get
$$
\sum_{\ell = 0}^{L} 
H^2(\P_0,\Q)|_{I_\ell}
\leq C_4 f^2 \sqrt{k},
$$
for some constant $C_4 > 0$.
\begin{claim}[Border regime: Hellinger distance]\label{claim:border-regime}
	$$
	H^2(\P_0,\Q)|_{J_0} = O(f^2 \sqrt{k}).
	$$
\end{claim}

\paragraph{Wrapping up}
We now prove Proposition~\ref{prop:main1}. 

\begin{proof}[Proof of Proposition~\ref{prop:main1}]
\[
H^2(\P_0,\Q) \leq H^2_{J_1}(\P_0,\Q)+ H^2_{J_0' \cup J_1' \cup J_2'}(\P_0,\Q) + H^2_{J_0}(\P_0,\Q)  \leq O(f^2 \sqrt{k}),
\]
by Claims~\ref{claim:high-substitution},~\ref{claim:low-substitution} and~\ref{claim:border-regime}. That implies Proposition~\ref{prop:main1}.
\end{proof}

\section{Matching upper bound}
\label{section:upper}

We give two proofs of the upper bound. 

\subsection{Proof of Theorem~\ref{thm:main2}}

\begin{proof}
	We use~\eqref{eq:testing-error} and construct an explicit 
	test $A$ as follows:
	\begin{itemize}
		\item Let $W$ be the number of genes such that
		$\theta/k \leq p_0$. Let
		$w = \P_0[\theta/k \leq p_0]$, $w' = \Q[\theta/k \leq p_0]$ and
		$$
		w^* 
		= m\frac{w + w'}{2}
		= mw + \frac{m}{2}(w' - w) 
		= mw' - \frac{m}{2}(w' - w).
		$$
		We consider the following event
		$$
		A = \{W \geq w^*\}.
		$$
	\end{itemize}
	\noindent It remains to show that the event $A$ is highly unlikely under $\P_0^{\otimes m}$ while
	being highly likely under $\Q^{\otimes m}$.
	We do this by bounding the difference $w' - w$
	and applying Chebyshev's inequality to $W$.
	
	Note that $W \sim \mathrm{Bin}(m,w)$ under $\P_0^{\otimes m}$ and
	$W \sim \mathrm{Bin}(m,w')$ under $\Q^{\otimes m}$. 
	By Assumption A1, $X \in [p_0 - \phi_f,p_0]$ under $\P_1$. By the Berry-Esseen
	theorem (e.g.~\cite{Durrett:96}),
	\begin{eqnarray}
	\P_1[\theta/k \leq p_0] \geq \E_1[\P_1[\theta \leq kX\,|\,X]] = \frac{1}{2} - O\left(\frac{1}{\sqrt{k}}\right)
	\geq \frac{1}{3},\label{eq:berry}
	\end{eqnarray}
	for $k$ large enough.
	Hence,	
	\begin{eqnarray}
	w' &=& \sigma_f \P_1[\theta/k \leq p_0] + (1 - \sigma_f) w\nonumber\\
	&\geq& \frac{1}{3} \sigma_f + (1- \sigma_f)w, \label{eq:wp1}
	\end{eqnarray}
	whereas by the computations in the previous
	section (more specifically, by summing over $\ell$ in~\eqref{eq:jpintegrator})
	\begin{equation}\label{eq:w}
	w = O\left(\frac{1}{\sqrt{k}}\right),
	\end{equation}
	and, similarly, since $f \sqrt{k} = o(1)$
	\begin{equation}\label{eq:wp2}
	w' = O\left(\frac{1}{\sqrt{k}}\right),
	\end{equation}
	from~\eqref{eq:wp1} and~\eqref{eq:w}.
	Consequently
	\begin{equation}\label{eq:gap}
	w' - w \geq \sigma_f \left(\frac{1}{3} - w\right) = \Omega(f).
	\end{equation}
	By Chebyshev's inequality,
	$$
	\P_0^{\otimes m}\left[A\right]
	\leq \frac{4m w (1 - w)}{m^2 (w' - w)^2}
	= O\left(\frac{1}{m f^2 \sqrt{k}}\right)
	\leq \frac{\delta}{2},
	$$
	for $c'$ large enough, where we used~\eqref{eq:w}
	and~\eqref{eq:gap}.
	Similarly, 
	$$
	\Q^{\otimes m}\left[A^c\right]
	\leq \frac{4m w' (1 - w')}{m^2 (w' - w)^2}
	\leq \frac{\delta}{2}.
	$$
\end{proof}

\subsection{Agnostic version}

Although Theorem~\ref{thm:main2} shows that our
bound in Theorem~\ref{thm:main1} is tight,
it relies on a test (i.e., the set $A$) that assumes knowledge of the
null and alternative hypotheses. Here we relax this assumption.

\paragraph{Pairwise distance comparisons}
We assume that we have two (independent) collections of
genes, $\tcal_1$ and
$\tcal_2$, one from each model, $\P_0$ and $\Q$
as in the previous section. 
We split the genes into two
equal-sized disjoint sub-collections, 
$(\tcal_1^1, \tcal_1^2)$ and
$(\tcal_2^1, \tcal_2^2)$. 
Assume for convenience that the total
number of genes is in fact $2m$ for each dataset. 
Let $C > 0$
be a constant, to be determined later (in equation~(\ref{eq:defC})). 
We proceed in two steps. 
\begin{enumerate}
	\item We first compute $\hat{p}_1$ and $\hat{p}_2$, 
	the $\frac{C}{\sqrt{k}}$-quantiles
	based on $\tcal_1^1$ and $\tcal_2^1$ respectively.
	Let $\hat{p} = \max\{\hat{p}_1,\hat{p}_2\}$. 
	
	\item Compute the fraction of genes, $\hat{w}_1$
	and $\hat{w}_2$, with $\theta/k \leq \hat{p}$ 
	in $\tcal_1^2$ and $\tcal_2^2$ respectively. 
	
\end{enumerate}
We infer that the first dataset comes from $\P_0^{\otimes 2m}$
if $\hat{w}_1 < \hat{w}_2$, and vice versa.
\begin{remark}
	Simply comparing the $\frac{C}{\sqrt{k}}$-quantiles 
	breaks down when
	$f \ll \frac{1}{k}$, 
	as it is quite possible that the quantiles will be identical since they can only take $k$ possible values. 
	However, even if the quantiles are identical, the probability of a gene being lower than the quantile is bigger if the distance is smaller. 
	This explains the need for the second phase in our algorithm. We remark further that the partition of the data into two sets is used for analysis purposes as it allows for better control of dependencies. 
\end{remark}

We show that this approach succeeds with probability
at least $1 - \delta$
whenever $m \geq c' \frac{1}{f^2 \sqrt{k}}$,
for $c'$ large enough. 
This proceeds from a series
of claims.
\begin{claim}[$\hat{p}$ is close to $p_0$]\label{claim:algo1}
	For $c'$ large enough,
	there is $C_1 > 0$ such that
	\begin{equation}\label{eq:quantile}
	\hat{p} \in \left[
	p_0, p_0 + \frac{C_1}{\sqrt{k}}
	\right]
	\end{equation}
	with probability $1 - \delta/2$.
\end{claim}
\begin{proof}
	The argument is similar to that in the proof
	of 
	Theorem~\ref{thm:main2}.
	
	By summing over $\ell$ in~\eqref{eq:jpintegrator},
	$$
	\P_0[\theta/k \leq p_0] \leq \frac{C'_1}{\sqrt{k}},
	$$
	for some $C_1' > 0$.
	For any $C''_1 > C'_1$, there is $C_1 > 0$ such that
	\begin{eqnarray*}
		\P_0\left[\theta/k \leq p_0 + \frac{C_1}{\sqrt{k}}\right]
		&\geq& \P_0\left[\theta/k \leq p_0 + \frac{C_1}{\sqrt{k}}\,\bigg|\,
		X\in\left[p_0, p_0 + \frac{C_1}{\sqrt{k}}\right]\right]\\
		&& \qquad\qquad\qquad\qquad \times \,\P_0\left[
		X\in\left[p_0, p_0 + \frac{C_1}{\sqrt{k}}
		\right]\right]\\
		&\geq& \frac{1}{3}\frac{\rho\, C_1}{\sqrt{k}} \geq \frac{C_1''}{\sqrt{k}}
	\end{eqnarray*}
	by the Berry-Esseen theorem (as in~\eqref{eq:berry}),
	for $C_1$ large enough.
	
	Let 
	\begin{equation} \label{eq:defC}
	C = \frac{C'_1 + C''_1}{2}.
	\end{equation}
	Let $W$ be the number of genes (among $m$) such that
	$\theta/k \leq p_0$ and $w = \P_0[\theta/k\leq p_0]$.
	Repeating the calculations in the proof of 
	Theorem~\ref{thm:main2}, 
	$$
	\P_0^{\otimes m}\left[W \geq m \frac{C}{\sqrt{k}}\right]
	\leq \frac{4m w (1 - w) k}{m^2 (C- C_1')^2}
	= \frac{1}{m}O\left(\sqrt{k}\right)
	\leq \frac{1}{c'} O (f^2 k) \leq \frac{\delta}{8},
	$$
	for $c'$ large enough.
	Similarly,
	let $\tilde{W}$ be the number of genes such that
	$\theta/k \leq p_0 + C_1/\sqrt{k}$ and $\tilde{w} = \P_0[\theta/k\leq 
	p_0 + C_1/\sqrt{k}]$. Then
	$$
	\P_0^{\otimes m}\left[\tilde{W} \leq m \frac{C}{\sqrt{k}}\right]
	\leq \frac{\delta}{8}.
	$$
	That implies that with probability $1 - \delta/4$ the
	$C/\sqrt{k}$-quantile under $\P_0^{\otimes m}$
	lies in the interval $[p_0, p_0 + \frac{C_1}{\sqrt{k}}]$.
	By monotonicity, 
	$\P_1[\theta/k\leq 
	p_0 + C_1/\sqrt{k}] \geq \tilde{w}$,
	and we also have
	$$
	\Q^{\otimes m}\left[\tilde{W} \leq m \frac{C}{\sqrt{k}}\right]
	\leq \frac{\delta}{8},
	$$
	which implies the claim.
\end{proof}
\begin{claim}[Test]
	For $c'$ large enough, if $\tcal_1$ comes from
	$\P_0^{\otimes 2m}$, $\tcal_2$ comes from $\Q^{\otimes 2m}$ and~\eqref{eq:quantile} holds, then
	$$
	\hat{w}_1 < \hat{w}_2
	$$
	with probability $1 - \delta/2$, and vice versa.
\end{claim}
\begin{proof}
	The proof is identical to that of Theorem~\ref{thm:main2}
	with $W$ now being the number of genes such that
	$\theta/k \leq \hat{p}$,
	$w = \P_0[\theta/k \leq \hat{p}]$, $w' = \Q[\theta/k \leq \hat{p}]$,
	and~\eqref{eq:w} and~\eqref{eq:wp2} now following 
	from Claim~\ref{claim:algo1} together with~\eqref{eq:jpintegrator} and~\eqref{eq:jintegrator}.
\end{proof}

\paragraph{Triplet reconstruction}
Consider again the three possible species trees
depicted in Figure~\ref{fig:3-species}.
By comparing the pairs two by two as described
in the agnostic algorithm, we can determine
which is the correct species tree topology.
Such ``triplet'' information is in general
enough (assuming the molecular clock hypothesis) 
to reconstruct a species tree
on any number of species (e.g.~\cite{SempleSteel:03}). We leave out the
details.


\section*{Acknowledgments}

We thank Gautam Dasarathy and Rob Nowak
for helpful discussions.

\bibliographystyle{alpha}
\bibliography{thesis,own,RECOMB12}

\end{document}